\documentclass{amsart}

\usepackage[pagewise]{lineno}

\usepackage{fancyhdr}
\usepackage{ifthen}
\usepackage{graphicx}
\usepackage[all]{xy}
\usepackage{tikz}
\usepackage{float}
\usepackage{amsmath}
\usepackage{amssymb}
\usepackage[utf8]{inputenc}
\usepackage{amsthm}
\usepackage{arcs}
\usepackage{mathrsfs}
\usepackage{textcomp}
\usepackage{times}
\usepackage[T1]{fontenc}
\usepackage{latexsym,amsfonts}
\usepackage{lipsum}
\usepackage{amsmath, amssymb, amsthm}
\usepackage{mathrsfs}
\usepackage{enumerate}
\usepackage{lmodern}
\usepackage{mathpazo}
\usepackage{euscript}
\usepackage{url}

\usepackage[utf8]{inputenc}
\usepackage[french]{minitoc}
\usepackage{mathtools}
\usepackage{outlines}

\newtheorem{theorem}{Theorem}[section]
\newtheorem*{theorem*}{Theorem 1.3}
\newtheorem{lemma}[theorem]{Lemma}
\newtheorem{corollary}[theorem]{Corollary}
\newtheorem{proposition}[theorem]{Proposition}

\theoremstyle{definition}
\newtheorem{definition}[theorem]{Definition}

\theoremstyle{remark}
\newtheorem{remark}[theorem]{Remark}

\theoremstyle{question}
\newtheorem{question}[]{Question}

\numberwithin{equation}{section}

\newcommand{\Rm}{\mathbb{R}^m}
\newcommand{\tg}{\tilde{g}}
\newcommand{\tu}{\tilde{u}}

\newcommand{\tx}{\tilde{x}}
\newcommand{\tf}{\tilde{f}}
\newcommand{\tfrak}{\mathfrak{t}}
\newcommand{\tfraks}{\mathfrak{t}^*}
\newcommand{\polytope}{{\Delta}}
\newcommand{\polytopepk}{\polytope_{p,k}}

\newcommand{\tphi}{\tilde{\phi}}

\newcommand{\tL}{\tilde{\L}}
\newcommand{\tLk}{\tilde{\L}_k}
\renewcommand{\L}{\textbf{L}}
\newcommand{\Lpk}{\L_{p,k}}
\newcommand{\tLL}{\tilde{L}}
\newcommand{\tomega}{\tilde{\omega}}
\newcommand{\tdelta}{\tilde{\delta}}
\newcommand{\tpolytope}{\tilde{\polytope}}
\newcommand{\tpolytopepk}{\tilde{\polytope}_{p,k}}
\newcommand{\T}{T}
\newcommand{\lattice}{\Lambda}
\renewcommand{\tt}{\textbf{t}}
\newcommand{\tJ}{\tilde{J}}
\newcommand{\tF}{\tilde{F}}
\newcommand{\fext}{\zeta_{(\polytope,\L,f)}}

\newcommand{\fextpmk}{\zeta_{(\polytopepk,\L,f^\pm_{p,k})}}
\newcommand{\fpmb}{f^\pm_b}
\newcommand{\fpmbm}{f^\pm_{-b}}
\newcommand{\wext}{\zeta_{(\polytope,\L,f,w)}}

\newcommand{\fDonFutinv}{\mathcal{F}_{\polytope,\L,f}}
\newcommand{\wfDonFutinv}{\mathcal{F}_{\polytope,\L,f,m+2}}
\newcommand{\tDonFutinv}{\mathcal{F}_{\tpolytope,\tL}}
\newcommand{\fwDonFutinv}{\mathcal{F}_{\polytope,\L,f,w}}
\newcommand{\ext}{\zeta_{(\polytope, \L)}}
\newcommand{\extt}{\zeta_{(\tpolytope, \tL)}}
\newcommand{\exttk}{\zeta_{(\tpolytopepk, \tLk)}}
\newcommand{\ckemconstant}{c_{(\polytope,\L,f)}}

\newcommand{\symppot}{\mathcal{S}(\polytope,\L)}
\newcommand{\tsymppot}{\mathcal{S}(\tpolytope,\tilde{\L})}
\newcommand{\tsigma}{\tilde{\sigma}}
\newcommand{\G}{\mathbb{T}}
\newcommand{\g}{\mathfrak{t}}

\newcommand{\GG}{\mathbb{G}}
\newcommand{\tGG}{\tilde{\mathbb{G}}}
\newcommand{\tpartial}{\tilde{\partial}}
\newcommand{\ta}{\tilde{a}}
\newcommand{\ttt}{\tilde{\tt}}
\renewcommand{\tJ}{\tilde{J}}
\newcommand{\tHH}{\tilde{\mathbb{H}}}
\newcommand{\HH}{\mathbb{H}}
\newcommand{\fpmpk}{f^\pm_{p,k}}
\newcommand{\fpmpun}{f^\pm_{p,1}}

\newcommand{\CP}{\mathbb{CP}}
\newcommand{\R}{\mathbb{R}}
\newcommand{\F}{\mathbb{F}}
\renewcommand{\P}{\mathbb{P}}

\begin{document}

\title{Conformally K\"ahler, Einstein-Maxwell metrics on Hirzebruch Surfaces}

\author{ISAQUE VIZA DE SOUZA}
\address{Département de Mathématiques, Université du Québec à Montréal}

\email{viza\underline{\hspace{.05in}}de\underline{\hspace{.05in}}souza.isaque@courrier.uqam.ca}






 \begin{abstract}
In this note we prove that a special family of Killing potentials on certain Hirzebruch complex surfaces, found by Futaki and Ono \cite{futaki2018volume}, gives rise to new conformally K\"ahler, Einstein-Maxwell metrics. The correspondent K\"ahler metrics are ambitoric \cite{Apostolov2016, Apostolov2019conformally} but they are not given by the Calabi ansatz \cite{lebrun2016einstein}. This answers in positive questions raised in \cite{futaki2018volume, futaki2019conformally}.
 \end{abstract}

\maketitle


\section{Introduction}

In this paper, we study the existence of conformally K\"ahler,  Einstein-Maxwell metrics on compact Hirzebruch complex surfaces. 

\begin{definition} \label{ckemdef}
A \emph{conformally K\"ahler,  Einstein-Maxwell} (cKEM for short) (real) $4$-dimensional manifold $(M, J, \tg)$ is a compact complex K\"ahler manifold $(M,J)$ with a Hermitian metric $\tg$ for which there exists a function $f$ such that $g=f^2\tg$ is a K\"ahler metric, satisfying also the following curvature conditions:
\begin{center}
 \begin{itemize}
     \item[(i) ] $Ric^{\tg}(J\cdot,J\cdot)=Ric^{\tg}(\cdot,\cdot)$;
     \item[(ii)] $Scal(\tg)=const$;
 \end{itemize}
\end{center}
 where $Ric^{\tg}$ and $Scal(\tg)$ denote the Ricci tensor and the scalar curvature of $\tg$.
\end{definition}
We shall refer to such Hermitian metrics as \emph{cKEM} metrics on $(M,J)$. When $M$ is a (real) 4-dimensional manifold, a cKEM metric provides a Riemannian signature analogue of a solution to the Einstein-Maxwell equations studied in General Relativity (see \cite{Apostolov2016, debever1984exhaustive, lebrun2016einstein, plebanski1976rotating}).

This  class of  Hermitian metrics on $4$-manifolds was first introduced by C.~LeBrun \cite{lebrun2008einstein}, who observed that they extend naturally the more familiar classes of K\"ahler metrics of constant scalar curvature (cscK for short) much studied since the pioneering work of E. Calabi \cite{calabi1982extremal, calabi1985extremal}, as well as the Einstein-Hermitian 4-manifolds classified in the compact case by LeBrun \cite{lebrun2012einstein}. The theory of cKEM metrics was consequently extended to arbitrary dimension by Apostolov-Maschler \cite{Apostolov2019conformally} who have also formulated the existence problem for such metrics  on a compact K\"ahler manifold in the framework of Calabi’s original approach of finding distinguished representatives for K\"ahler metrics in a given de Rham class. The point of view of \cite{Apostolov2019conformally} was  generalized by A.~Lahdili \cite{lahdili2019kahler} who showed that the K\"ahler metrics giving rise to cKEM Hermitian structures arise as a special case of a more general notion of \emph{weighted constant scalar curvature K\"ahler} metrics to which a great deal of the known machinery in the cscK case can be effectively applied. Finally, additional motivation for studying conformally K\"ahler Einstein-Maxwell 4-manifolds came from the recent realization by Apostolov-Calderbank \cite{apostolov_cr_2020} that such metrics give rise to extremal Sasaki structures on $5$-manifolds \cite{Boyer2009sasaki}.

With the above motivation in mind, the existence theory for cKEM metrics is rapidly taking shape. Families of non-trivial examples were constructed on $\F_0=\CP^1 \times \CP^1$ \cite{lebrun2015einstein} and on the Hirzebruch complex surfaces $\F_k= \P(\mathcal{O} \oplus \mathcal{O}(k)) \to \CP^1$ $,k>0$, \cite{lebrun2016einstein} by C.~LeBrun.  An extension of these constructions to other ruled complex surfaces appears in \cite{koca2016strongly}. LeBrun's examples on $\F_k$ have large groups of automorphisms; actually they are of cohomogeniety one under the action of suitable compact groups. It was shown in \cite{lahdili2019automorphisms, futaki2019conformally} that any K\"ahler metric on $\F_k$ which is conformal to an Einstein-Maxwell Hermitian metric must be invariant under the action of a $2$-dimensional torus, i.e. it is toric. Toric cKEM metrics have been studied more generally in \cite{Apostolov2019conformally} and as a consequence of this work it was realized that the existence of a K\"ahler metric conformal to an Einstein-Maxwell Hermitian one in a given K\"ahler class on $\F_k$  can be characterized in terms of the corresponding Delzant image (which is a Delzant trapezoid $\polytope \subset \R^2$)  as follows:  

\begin{itemize}\label{abcondition}
\item[(a)]  there exists an affine linear function $f$ on $\R^2$ which is positive on $\polytope$ and satisfies a non-linear algebraic condition, and 
\item[(b)] a certain linear functional depending on $f$ is strictly positive on convex piecewise affine linear functions over $\polytope$ which are not affine linear.  
\end{itemize}

The condition (a) is characterized in \cite{Apostolov2019conformally} as the vanishing of a Futaki-like invariant  on $M$ whereas the condition (b) is referred there as \emph{$f$-$K$-stability} of the pair $(\polytope,f)$.  It is shown in \cite{Apostolov2019conformally, futaki2018volume} that on $\F_0$, (a) holds only for the affine linear functions associated to the  explicit solutions found in \cite{lebrun2015einstein}, thus leading to a complete classification of cKEM metrics on $\F_0$.  Furthermore, \cite{futaki2018volume} simplifies the search for solutions of (a) by interpreting them  as critical points of a volume functional. In particular, \cite{futaki2018volume} essentially identifies all solutions of (a) on the first Hirzebruch surface $\F_1$. Their analysis reveals that certain K\"ahler classes on $\F_1$ admit two additional positive affine linear functions  $f^+$ and $f^-$ satisfying (a), which do not correspond to the solutions  found in \cite{lebrun2015einstein}. However, even though \cite{futaki2020existence} provides numerical evidence that the condition (b) for those solutions $f^+$ and $f^-$ of (a) holds true, the question of whether or not $f^\pm$ do actually correspond to (new) cKEM metrics on $\F_1$ was left open. One of the purposes of this article is to give a positive answer to this question.

\begin{theorem} \label{thm1}
The first Hirzebruch surface  $\F_1$ admits conformally Einstein-Maxwell, toric Kähler metrics which are regular ambitoric of hyperbolic type in the sense of \cite{Apostolov2016}. These, together with the metrics of Calabi type  constructed by LeBrun in \cite{lebrun2016einstein} are the only conformally Einstein-Maxwell K\"ahler metrics on $\F_1$, up to a holomorphic homothety.
\end{theorem}

We note that in \cite{futaki2018volume}, it is shown that similar solutions  $f_k^+$ and $f_k^-$ of the condition (a) also arise on any Hirzebruch surface  $\F_k$, $2\leq k\leq 4$, but it is unknown if these, together with the affine linear functions, corresponding to the solutions in \cite{lebrun2015einstein} are the only solutions. Our method of proof also yields

\begin{theorem} \label{thm2}
Each Hirzebruch surface $\F_k$, with $k=1,2,3,4$, admits  conformally Einstein-Maxwell, toric K\"ahler metrics which are regular ambitoric of hyperbolic type.
\end{theorem}

We now briefly explain the main idea of the proof of the above results. The essential observation \cite{apostolov_cr_2020} is that if $f$ is a positive affine linear function over a Delzant polytope $(\polytope,\L)$ in $\R^n$, one can  associate to $(\polytope,\L,f)$ a different labelled compact convex simple polytope $(\tpolytope, \tL)$ in $\R^n$, called the $f$-twist trasform of $(\polytope,\L)$. Following the theory in \cite{apostolov_cr_2020}, we observe in Proposition \ref{keyremark} below that $(\polytope,\L)$ is $f$-$K$-stable  (i.e. (b) holds with respect to  $f$) if and only if  $(\tpolytope, \tL)$ is (relatively) $K$-stable in the sense originally introduced by Donaldson \cite{donaldson2002scalar} in the cscK case and by \cite{zhou2008stability} in general. Thus, proving Theorems \ref{thm1} and \ref{thm2} above reduces (via \cite[Theorem 5]{Apostolov2019conformally}) to checking that the  corresponding $f^\pm_{k}$-twists $(\tpolytope_{k}^\pm, \tL_{k}^\pm)$ of the Delzant trapezoids associated to $\F_k$  are $K$-stable (see Theorem \ref{equivthm} for a precise statement). Our key new observation here is that $(\tpolytope_{k}^\pm, \tL_{k}^\pm)$ are in fact non-trapezoidal equipoised labelled quadrilaterals in the sense of E. Legendre \cite{legendre2011toric}. It then follows from the latter work that the $K$-stability of $(\tpolytope_{k}^\pm, \tL_{k}^\pm)$ can be reduced to checking the positivity of two polynomials of degree $\leq 4$ over the given intervals. This is shown to hold for any non-trapezoidal equipoised quadrilateral in \cite[Example 1]{apostolov_cr_2020}, thus concluding the proof of the existence. The uniqueness in Theorem \ref{thm1} follows from the fact that any cKEM metric must be invariant under the action of a maximal torus in the automorphism group of $\F_k$ \cite{futaki2020existence, lahdili2019kahler}, the uniqueness result for toric cKEM metrics established in \cite{Apostolov2019conformally}, and explicit computations in \cite{futaki2018volume}.

The paper is organized as follows. In Sections 2 and 3,  we fix the notations
and, following \cite{Apostolov2019conformally} and \cite{futaki2018volume},  introduce the Calabi type problem of the search of K\"ahler metrics conformally related to an Einstein--Maxwell hermitian metric  and, more generally, of weighted extremal K\"ahler metrics,  in a given coholomology class. In Section 4,  we specialize to the toric case, and, following \cite{Apostolov2019conformally}, express the problem in terms of the corresponding Delzant polytope. In Section 5,  we  recall  the notion of $f$-twist of a labelled polytope introduced in \cite{apostolov_cr_2020},  and use it to relate the existence of weighted extremal metrics and extremal metrics as well as the corresponding notions of stability of the corresponding polytopes. In Section 6,  we make use of the results of Section 5 together with the theory of ambitoric geometry \cite{Apostolov2015, legendre2011toric} to prove Theorems \ref{thm1} and \ref{thm2}.

\section*{Acknowledgements}

This paper is part of the author\textquotesingle s Ph.D. thesis. The author would like to thank his thesis supervisor Vestislav Apostolov for his invaluable advice and for sharing his insights with him. The author is also grateful to Eveline Legendre, who kindly shared with him her observations used in the proof of Theorem \ref{equipoisedresult}. He would also like to thank Abdellah Lahdili and Lars Martin Sektnan for enlightening discussions, the Federal University of Ouro Preto and Université du Québec à Montréal for their financial support, and the referee for his careful reading and suggestions that greatly improved the text.

\section{Conformally Kähler, Einstein-Maxwell Geometry}

We start by recalling some properties of cKEM metrics, according to C.~LeBrun \cite{lebrun2016einstein} and Apostolov-Calderbak-Gauduchon \cite{Apostolov2016}. Our notation will closely follow that of \cite{Apostolov2019conformally, futaki2018volume}.

 Let $\tg$ be a Hermitian metric on a compact complex K\"ahler manifold $(M,J)$ satisfying Definition \ref{ckemdef}. 

As the Ricci tensor $Ric^{g}$ of the K\"ahler metric $g=f^2\tg$ also satisfies $Ric^{\tg}(J\cdot,J\cdot)=Ric^{\tg}(\cdot,\cdot)$, and 
\begin{equation} \label{ricci}
    Ric^{\tg}= Ric^{g}+\frac{2m-2}{f}D^gdf+hg,
\end{equation} where $D^g$ denotes de Levi-Civita connection of $g$ and $h$ is a smooth function not given explicitly, the condition $(i)$ in Definition \ref{ckemdef} is equivalent to the condition that the vector field $K=Jgrad_gf$ is Killing for both $g$ and $\tg$. Furthermore, condition $(ii)$ in Definition \ref{ckemdef} reads as
\begin{equation} \label{scal}
    Scal(\tg)=f^2Scal(g)-2(2m-1)f\Delta_gf-2m(2m-1)|df|_g^2=c
\end{equation}
where $c$ is a constant, $\polytope_g$ is the Riemannian Laplacian of $g$ and $Scal(g)$ is the scalar curvature of $g$. We define  the function $$Scal_f(g):=f^2Scal(g)-2(2m-1)f\Delta_gf-2m(2m-1)|df|_g^2,$$ and refer to it as the $(f,2m)$-scalar curvature of $g$. This is a particular case (with $w=2m$) of the notion of  $(f,w)$-\emph{scalar curvature} $$Scal_{(f,w)}(\tg):=f^2Scal(g)-2(w-1)f\Delta_gf-w(w-1)|df|_g^2 $$studied in \cite{Apostolov2019conformally, lahdili2019automorphisms} for an arbitrary real number $w$.

Thus, every cKEM metric admits a Killing vector field $K:=Jgrad_{g}f$, and we know from \cite[Theorem 1]{lahdili2019automorphisms} and \cite[Theorem 2.1]{futaki2020existence} that every cKEM metric on a compact manifold is invariant under the action of a maximal compact real torus $\G$ inside the reduced automorphism group $ Aut_r(M,J)$ of $(M,J)$ with $K\in\g= \mathrm{Lie}(\G)$ (see \cite{gauduchoncalabi} for the definition of $Aut_r(M,J)$). More precisely:

\begin{theorem}[\cite{futaki2020existence, lahdili2019automorphisms}] \label{lahdili1} Let $(M,g,J)$ be a compact K\"ahler manifold and $K=Jgrad_gf$ a Killing vector field with positive Killing potential $f$. If $g$ is $f$-extremal (i.e. if $Scal_f(g)$ is a Killing potential) then $g$ is invariant under the action of a maximal compact real torus $\G\subset Aut_r(M,J)$ such that $K$ and $Jgrad_g(Scal_f(g))$ belong to $\mathrm{Lie}(\G)$.
\end{theorem} 

\section{The Weighted Calabi Problem}

Now we fix a maximal compact torus $\G\subset Aut_r(M,J)$, and a vector field $K\in\mathfrak{\g}:=\mathrm{Lie}(\G)$. Let $\omega_0$ be a $\G$-invariant K\"ahler form, and $ \Omega=[ \omega_0]\in H^2_{DR}(M,\mathbb{R})$ be a fixed K\"ahler class. The problem we are going to study is to find a $\G$-invariant K\"ahler metric $g$ with K\"ahler form $\omega_g\in\Omega$, such that $\tg=f^{-2}g$ is a cKEM metric, for $f>0$ such that $Jgrad_gf=K$.

Denote by $\mathcal{K}^{\G}_{\Omega}$ the space of $\G$-invariant Kähler metrics $g$ on $(M,J)$ with $\omega_g\in\Omega$. Then the vector field $K\in\g$ is Hamiltonian with respect to $\omega_g$ (see \cite[Chapter 2]{gauduchoncalabi}), i.e. $$\iota_K\omega_g=-df_{K,g}$$ for a smooth function $f_{K,g}$ on $M$. Such a function is called a \textit{Killing potential} of $K$ with respect to $\omega_g$. We observe that this function is defined up to an additive constant, so we further fix the setting by requiring $$\int_Mf_{K, g}\frac{\omega}{m!}=a,$$ where $a$ is a fixed real constant. We shall denote by $f_{K,a,g}$ the unique function satisfying the above relations.

Since $min\left\{ f_{K,a,g}|x\in M\right\}$ is independent of $g$ in $\mathcal{K}^{\G}_{\Omega}$ (see e.g. \cite[Lemma 1]{Apostolov2019conformally}), following \cite{futaki2018volume}, we define:

\begin{equation} \label{pgomega}
    \mathcal{P}^{\G}_{\Omega}:=\left\{(K,a)\in\mathfrak{\g}\times\mathbb{R}|f_{K,a,g}>0 \right\},
\end{equation}
\begin{equation} \label{hgomega}
    \mathcal{H}^{\G}_{\Omega}:=\left\{\tg_{K,a}=\frac{1}{f^2_{K,a,g}}g \left| (K,a)\in\mathcal{P}^{\G}_{\Omega},\right.g\in\mathcal{K}^{\G}_{\Omega} \right\}.
\end{equation}

From now on we identify the K\"ahler metric $g$ with its Kähler form $\omega_g$, and we drop the subscript $g$. Fixing $(K,a)\in\mathcal{P}^{\G}_{\Omega}$, let
\begin{equation} 
\label{hgomegaka}
    \mathcal{H}^{\G}_{\Omega,K,a}:=\left\{\tg_{K,a}| g\in\mathcal{K}^{\G}_{\Omega} \right\}
\end{equation}
and

\begin{equation} \label{constfut1}
    c_{\Omega,K,a}:=\left(\int_M s_{\tilde{g}_{K,a}}\frac{1}{f_{K,a,g}^{2m+1}}\frac{\omega^m}{m!}\right)\Big/\left( \int_M \frac{1}{f_{K,a,g}^{2m+1}}\frac{\omega^m}{m!}\right).
\end{equation}
It follows from \cite[Corollary 1]{Apostolov2019conformally} that $c_{\Omega,K,a}$ is a constant independent of the choice of $g\in\mathcal{K}^{\G}_{\Omega}$.

Also, for each vector field $H\in\g$ with Killing potential $f_{H,b,g}$, we consider
\begin{equation}\label{futakickem}
    \mathfrak{F}^{\G}_{\Omega,K,a}(H):=\int_M\left(\frac{s_{\tilde{g}_{K,a}}-c_{\Omega,K,a}}{f_{K,a,g}^{2m+1}}\right)f_{H,b,g}\frac{\omega^m}{m!},
\end{equation}
which according to \cite[Corollary 1]{Apostolov2019conformally} is a linear functional, independent of the choice of $(g,b)\in\mathcal{K}^{\G}_{\Omega}\times\mathbb{R}$.

\begin{definition} \label{ckemFutinv}
The linear map $\mathfrak{F}^{\G}_{\Omega,K,a}:\g\longrightarrow\mathbb{R}$ defined by (\ref{constfut1}) and (\ref{futakickem}) is called the \textit{cKEM-Futaki invariant}.
\end{definition}

\begin{theorem}[{\cite[Corollary 1]{Apostolov2019conformally}}] \label{futvanishing}
The vanishing of $\mathfrak{F}^{\G}_{\Omega,K,a}$ is an obstruction to the existence of a cKEM metric in $\mathcal{H}^{\G}_{\Omega,K,a}$.
\end{theorem}

 \begin{remark}
The main result in \cite{futaki2018volume} gives a useful characterization of the condition $\mathfrak{F}^{\G}_{\Omega,K,a}\equiv0$. Indeed, the authors prove that $\mathfrak{F}^{\G}_{\Omega,K,a}\equiv0$ if and only if $(K,a)$ is a critical point of the suitably normalized volume functional acting on $\mathcal{P}^{\G}_{\Omega}$. The usefulness of their theorem resides in the fact that it allows for a systematic computation of the vanishing of the cKEM-Futaki invariant.
 \end{remark}

 \section{Toric K\"ahler Manifolds}

From now on, we specialize to the toric case, i.e. we assume that $\G\subset Aut_r(M,J)$ is an $m$-dimensional torus, where $m$ is the complex dimension of $(M,\omega, J)$. We recall that by Theorem \ref{lahdili1}, any cKEM metric $\tilde{g}$ must be obtained from a toric K\"ahler metric $(g,\omega)$. This is the situation studied in \cite{Apostolov2019conformally}, by using the Abreu-Guillemin formalism \cite{Abreu1998kahler, guillemin1994kaehler}.

Let $(M,\omega, \G)$ be a compact symplectic toric manifold and $\mu:M\rightarrow\tfraks$ its moment map. It is well know \cite{Atiyah1982convexity, guillemin1982convexity} that the image of $M$ by $\mu$ is a compact simple convex polytope $\polytope\subset\tfraks$. Furthermore, it is shown in \cite{delzant1988hamiltoniens} that $\polytope$ can be given the structure of a \emph{labelled Delzant polytope} $(\polytope,\L)$, i.e. a compact convex simple polytope with $d$ facets, together with a set $\L=\{L_1,\ldots,L_d\}$ of non-negative affine linear functions $L_i$ defining $\polytope$ by 
$$\polytope:=\{x\in\tfraks: L_i(x)\geq0, i=1,\ldots,d\},$$
and such that $dL_i\in\tfrak$ are primitive elements of the lattice $\lattice\subset\tfrak$ of circle subgroups of $\G$ (\emph{integrality condition}). It also follows from \cite{delzant1988hamiltoniens} that the compact symplectic toric manifold $(M,\omega,\G)$ can be reconstructed from the corresponding labelled integral Delzant polytope $(\polytope,\L)$.

Now, let $(M,g, J, \G)$ be a compact toric K\"ahler manifold and $\mu:M\rightarrow\tfraks$ its moment map. According to \cite{guillemin1994kaehler}, on the dense open subset $M^0:=\mu^{-1}(\polytope^0)$ (where $\polytope^0$ denotes the interior of $\polytope$), the toric Kähler structure $(g,J,\omega)$ can be written in moment-angle coordinates $(x,\tt)$ as:
\begin{equation}\label{momentanglecoords}
    \begin{array}{lcl}
    g=\langle dx,\GG(x),dx\rangle+\langle d\tt,\HH(x),d\tt\rangle, & & Jd\tt=-\langle \GG(x),dx\rangle, \\
    \omega=\langle dx\wedge d\tt\rangle, & & Jdx=-\langle \HH(x),d\tt\rangle,
    \end{array}
\end{equation}
where $\HH$ is a smooth positive definite $S^2t^*$-valued function on the moment image $\polytope^0$ and $\GG=\HH^{-1}$ is its pointwise inverse, a smooth $S^2t$-valued function. Furthermore, $\GG=Hess(u)$ is the Hessian of a real function $u\in\mathcal{C}^{\infty}(\polytope^0)$, called symplectic potential of $(g,J,\omega)$.

We denote by $\symppot$ the set of \emph{symplectic potentials} of globally defined $\G$-invari-ant $\omega$-compatible K\"ahler metrics $(g,J)$ on $(M,\omega, \G)$. By the theory in \cite{Abreu1998kahler, Abreu2001} (see also \cite[Proposition 1]{Apostolov2004} and \cite{donaldson2009constant}), $\symppot$ consists of smooth strictly convex functions ${u\in\mathcal{C}^{\infty}(\polytope^0)}$, whose inverse Hessian 
\begin{equation*}
    \HH^u=(H_{ij}^u)=\left(\frac{\partial^2u}{\partial x_i\partial x_j}\right)^{-1}
\end{equation*}
 is smooth on $\polytope$, positive definite on the interior of any face and satisfies, for every $y$ in the interior of a facet $F_i\subset\Delta$ with inward normal $e_i=dL_i$, the following boundary conditions \cite[Proposition 1]{Apostolov2004}:
\begin{equation} \label{toriccond1}
\HH^u_{y}(e_i,\cdot)=0 \textrm{ and } d\HH_{y}^u(e_i,e_i)=2e_i.
\end{equation}

\begin{remark}
$\mathcal{S}(\polytope,\L)$ can be introduced independent of the integrality condition on $(\polytope,\L)$, as in \cite{donaldson2002scalar}.
\end{remark}

In \cite{Abreu1998kahler}, Abreu computed the scalar curvature of the metric (\ref{momentanglecoords}) associated to a symplectic potential $u\in\mathcal{S}(\polytope,\L)$ to be  the pull-back by the moment map of the smooth function on $\polytope$
\begin{equation} \label{toriccond2}
S(u)=-\sum_{i,j=1}^m\frac{\partial^2H_{ij}^u}{\partial x_i\partial x_j}.
\end{equation}
In the above formula and in what follows, we use the conventions of \cite{Abreu2001, Apostolov2004, donaldson2009constant} (see also Apostolov\textquotesingle s lecture notes \cite{apostolovlecturenotes}).
 
Notice that in the toric setting the space of Killing potentials of elements in $\g$ with respect to $(g,\omega)$ is in one-to-one correspondence with affine linear functions (pulled-back by $\mu$) on $\g^*$. The \emph{extremal affine linear function} $\ext$ is the $L^2$-projection (with respect to the euclidean measure) of $S(u)$ to the finite dimensional space of affine linear functions on $\tfraks$. In fact, $\ext$ is independent of the symplectic potential $u\in\mathcal{S}(\polytope,\L)$ (see \cite{donaldson2002scalar}) and may also be defined as the solution of a linear system depending only on $(\polytope,\L)$.

Any solution $u\in\symppot$ of 
\begin{equation}\label{abreueq}
S(u)=-\sum_{i,j=1}^m\frac{\partial^2H_{ij}^u}{\partial x_i\partial x_j}=\ext
\end{equation}
gives rise to an \emph{extremal K\"ahler metric} and (\ref{abreueq}) is know as the \emph{Abreu equation}. The cscK case reduces to the special situation when $\ext$ is constant.

In the case when $(M,\omega,J,\G)$ is a toric K\"ahler manifold and $f$ is an affine linear function on $\tfraks$ which is positive on $\polytope$, the scalar curvature of $\tg=f^{-2}g$ is computed in \cite{Apostolov2019conformally} to be 
\begin{equation} \label{fscal}
S_{\tg}(u)=-f^{2m+1}\sum_{i,j=1}^m\left(\frac{1}{f^{2m-1}}H_{ij}^u\right)_{,ij}.
\end{equation}

Closely related to the discussion above, it is proved in \cite{Apostolov2019conformally} that the $L^2$-projection of (\ref{fscal}) to the space of affine linear functions on $\tfraks$ is independent of $g$ (i.e. of $u\in\symppot$) and in the same way one can consider the following \emph{weighted Abreu equation} for $u\in\symppot$:
\begin{equation}\label{fextscal}
-f^{2m+1}\sum_{i,j=1}^m\left(\frac{1}{f^{2m-1}}H_{ij}^u\right)_{,ij}=\fext,
\end{equation}
where $\fext$ is defined in terms of $(\polytope,\L,f)$.

Solutions to the problem above are called \emph{$(f,2m)$-extremal K\"ahler metrics} and in the special case when  $\fext$ is constant, the metric $f^{-2}g$ is conformally K\"ahler, Einstein-Maxwell.

More generally, one can define \cite{Apostolov2019conformally, lahdili2019automorphisms} a \emph{$(f,w)$-extremal toric K\"ahler metric} as a solution of the equation
\begin{equation}\label{wextscal}
-f^{w+1}\sum_{i,j=1}^m\left(\frac{1}{f^{w-1}}H_{ij}^u\right)_{,ij}=\wext,
\end{equation}
for $u\in\symppot$, $f$ a positive affine linear function on $\polytope$, and $\wext$ an affine linear function determined by $(\polytope,\L,f,w)$.

\begin{theorem}[{\cite[Theorem 3]{Apostolov2019conformally}}] 
Any two solutions $u_1,u_2\in\mathcal{S}(\polytope,\L)$ of (\ref{wextscal}) differ by an affine linear function.  In particular, on a compact toric K\"ahler manifold $(M,\omega, J, \G)$, for any fixed positive affine linear function in momenta $f=f_{K,a,g}$, there exists at most one, up to a $\mathbb{T}$-equivariant isometry, $\omega$-compatible $\mathbb{T}$-invariant K\"ahler metric $g$ for which $\tilde{g}_{K,a}=f^{-2}g$ is a conformally K\"ahler, Einstein-Maxwell metric.
\end{theorem}

Similarly to the extremal toric case studied in \cite{donaldson2002scalar}, there exists an obstruction to find a solution to (\ref{wextscal}) which is called \emph{$(f,w)$-$K$-stability} of $(\polytope,\L,f)$, which we now explain following \cite{Apostolov2019conformally, lahdili2019kahler}.

\begin{definition}\label{definvfutaki}
The \emph{$(f,w)$-Donaldson-Futaki invariant} $\fwDonFutinv$ of a labelled compact simple convex polytope $(\polytope,\L)$ and a given positive affine linear function $f$ on $\polytope$ is defined by
\begin{equation} \label{invfutaki}
\fwDonFutinv(\phi)=2\int_{\partial\polytope}\frac{\phi}{f^{w-1}} d\sigma-\int_{\polytope}\frac{\phi}{f^{w+1}}\wext dx,
\end{equation}
where $dx$ is an euclidean measure on $\polytope$ and $d\sigma$ is a measure on any facet ${F_i\subset\polytope}$ defined by $dL_i\wedge d\sigma=-dx$. In the above formula, the affine linear function $\wext$ is the unique affine linear function such that $\fwDonFutinv(\phi)=0$ for all affine linear functions $\phi$ on $\polytope$.
\end{definition}

\begin{definition}
 A labelled polytope $(\polytope, \L)$ is \emph{$(f,w)$-$K$-stable} if the associated $(f,w)$-Donaldson-Futaki invariant $\fwDonFutinv$ is non-negative on any convex piecewise affine linear function $\phi$ on $\polytope$, and vanishes if and only if $\phi$ is affine linear.
\end{definition}

\begin{remark} Note that if we take $f\equiv 1$ in Definition \ref{definvfutaki}, then we recover the usual (relative) \emph{Donaldson-Futaki invariant} introduced in \cite{donaldson2002scalar, zhou2008stability}. Also, the $(f,2m)$-Donaldson-Futaki invariant, hereafter denoted by $\fDonFutinv$, is equal to $(2\pi)^{-m}$ times the Futaki invariant defined by (\ref{futakickem}), when restricted to functions $\phi$ which are affine linear in momenta. 
\end{remark}

\begin{theorem}[\cite{Apostolov2019conformally}] \label{am19thm}
If $\fext=c$ is constant and (\ref{fscal}) admits a solution $u\in\symppot$ then $(\polytope,\L)$ is $(f,2m)$-$K$-stable. 
\end{theorem}

To summarize, the existence of $g\in \mathcal{K}^\G_{\Omega}$ which is conformal to an Einstein-Maxwell Hermitian metric is equivalent to the existence of $u\in\symppot$ and a positive affine linear function $f$ on $\polytope$, satisfying (\ref{fextscal}). Moreover, if a solution exists then
\begin{equation}\label{abconditions}
\begin{split}
&(a)\text{ }\fext=c\text{  is constant};\\
&(b)\text{ }(\polytope,\L,f)\text{ is }(f,2m)\text{-}K\text{-stable};
\end{split}
\end{equation}

The constant $c$ in $(a)$ is prescribed by $(\polytope,\L,f)$, via the formula (\cite[Theorem 2]{Apostolov2019conformally}):

\begin{equation} \label{ckemconstant}
c=\ckemconstant :=2\frac{\int_{\partial\polytope}\frac{1}{f^{2m-1}}d\sigma}{\int_\polytope\frac{1}{f^{2m+1}}d\mu}.
\end{equation}
In particular, it is always positive.
It is not known at present whether or not $(a)$ and $(b)$ are sufficient in general, but  a positive answer is given in the special case when $(\polytope,\L)$ is a labelled quadrilateral.

\begin{theorem}[{\cite[Theorem 5]{Apostolov2019conformally}}]  \label{equivthm} 
Let $(M, \omega,\mathbb{T})$ be a compact symplectic toric 4-orbifold whose rational Delzant polytope is a labelled quadrilateral ($\polytope,\L$) and $f$ a positive affine linear function on $\polytope$ which satisfies $(a)$. Then $(b)$ is equivalent to the existence of a $\mathbb{T}$-invariant K\"ahler metric $g$ such that $\tilde{g}=f^{-2}g$ is a conformally K\"ahler, Einstein-Maxwell metric on $M$.
\end{theorem}

\section{The $f$-twist of a labelled polytope}

In this section we follow \cite{apostolov_cr_2020}, where the authors introduce the \emph{$f$-twist} transform of a labelled polytope. The interested reader can consult the original paper for more details. A special case of the correspondence was first seen in \cite{Apostolov2019conformally} (see Proposition 3) where a bijection between \emph{ambitoric Einstein-Maxwell metrics} and \emph{ambitoric extremal metrics} of positive scalar curvature was found. In \cite{apostolov_cr_2020}, the authors introduce the $f$-twist transform more generally in terms of a pair of K\"ahler metrics arising as transversal K\"ahler structures of Sasaki metrics compatible with the same CR structure and having commuting Sasaki-Reeb vector fields. This leads to an interesting general equivalence between cKEM and extremal K\"ahler metrics in real dimension 4, which is the case we are most interest in.

\begin{definition} \label{ftwistdef}
Let $\polytope$ be a polytope in $\Rm$ containing the origin and $f$ a positive affine linear function on $\polytope$. We define the $f$-\emph{twist} transform $\tpolytope$ of $\polytope$ to be the image of $\polytope$ under the change of variables $T(x)=\tx:=\frac{x}{f(x)}$ where $x=(x_1,\ldots,x_m)$ are the euclidean coordinates of $\Rm$, i.e. $\tx_i=\frac{x_i}{f(x)}$ for $i=1,\ldots,m$. We also define the $f$-twist transform of a function $\phi$ to be the function $\tphi(\tx):=\frac{\phi(x)}{f(x)}$.
\end{definition}

\begin{lemma}\label{afflinlemma}
Let $\phi(x)$ be an affine linear function in the coordinates $x=(x_1,\ldots,x_m)$ in $\Rm$, and $\tphi(\tx)=\frac{\phi(x)}{f(x)}$ its $f$-twist transform. Then $\tphi(\tx)$ is an affine linear function in the coordinates $\tx=(\tx_1,\ldots,\tx_m)$ in $\Rm$. In particular, if $(\polytope,\L)$ is a labelled polytope containing the origin and $f(x)$ is a positive affine linear function  on $\polytope$, then the $f$-twist transform of $(\polytope, \L)$, denoted by $(\tpolytope,\tL)$, is a labelled polytope with respect to the labeling $\tL:=\frac{\L}{f(x)}$ which contains the origin, i.e. $\tpolytope=T(\polytope)$ is defined by $\{\tLL_i(\tx)\geq0; i=1,\ldots,d\} $ where ${\tLL_i(\tx):=\frac{L_i(x)}{f(x)}}$.
\end{lemma}

\begin{remark}
When $(\polytope,\L)\subset\Rm$ is a rational Delzant polytope associated to a compact toric orbifold, we can assume without loss of generality that the origin is inside $\polytope$. The last claim of Lemma~\ref{afflinlemma} then follows from \cite{legendre2011existence} and the geometric interpretation of the $f$-twist transform given in \cite[Theorem 1 and Lemma 5]{apostolov_cr_2020}.
\end{remark}

\begin{proof}
Let $\phi(x)=b_0+b_1x_1+\cdots+b_mx_m$ be an affine linear function in the coordinates $(x_1,\ldots,x_m)$. We observe that
\begin{equation} \label{linear1}
\tphi(\tx)=\frac{b_0}{f(x)}+\sum_{i=1}^n\left(b_i\frac{x_i}{f(x)}\right)=\frac{b_0}{f(x)}+\sum_{i=1}^n b_i\tx_i.
\end{equation}
Also, for ${f(x)=a_0+\sum_{i=1}^na_ix_i}$, we have
\begin{equation} \label{linear2}
\frac{1}{f(x)}=\frac{1}{a_0}\left(1-a_1\tx_1-\ldots-a_n\tx_m\right).
\end{equation}
It follows from equations (\ref{linear1}) and (\ref{linear2}) that
\begin{equation*}
\tphi(\tx)=\frac{1}{a_0}\left( b_0+\sum_{i=1}^n \left(b_i-b_0a_i\right)\tx_i\right),
\end{equation*}
establishing the first part of the Lemma.

For the second part we observe that given the labelled Delzant polytope $(\polytope,\L)$ then $\polytope=\{x\in\Rm: L_i(x)\geq0, i=1,\ldots,d\}$. So $x\in\polytope$ if and only if $\tx\in\tpolytope$ or equivalently $\tpolytope:={\{\tx\in\Rm: \tLL_i(\tx)\geq0, i=1,\ldots,d\}=T(\polytope)}$.
\end{proof}

\begin{remark}\mbox{}
\begin{enumerate}
\item[(i)] Equation (\ref{linear2}) in the proof of the Lemma \ref{afflinlemma} defines a distinguished affine linear function in the new coordinates $\tx$, which hereafter we will denote by $$\tf(\tx):=\frac{1}{f(x)}=\frac{1}{a_0}\left(1-a_1\tx_1-\ldots-a_n\tx_m\right);$$ 

\item[(ii)] For a given affine linear function $\phi$ defined on $\polytope$, we have 
\begin{equation} \label{twistpullbackrelation}
\begin{aligned}
\phi(x)=(\T^\ast\phi)(\tx)=\frac{\tilde{\phi}(\tx)}{\tf(\tx)}.
\end{aligned}
\end{equation}
\end{enumerate}

\end{remark}

For a symplectic potential $u\in\symppot$ we consider the $f$-twist transform of $u$ defined by 
\begin{equation}\label{potentialtwist}
\tu(\tx):=\frac{u(x)}{f(x)}.
\end{equation}
Then we have

\begin{lemma} \label{compactificationlemma}
If $u\in\symppot$, then $\tu\in\tsymppot$.
\end{lemma}

In the case when $(\polytope,\L)$ is rational, this result compared with \cite[Theorem 1]{apostolov_cr_2020} and Lemma \ref{afflinlemma}, yields the claim in Lemma \ref{compactificationlemma}. Here we give a general argument for the sake of completeness. In order to prove Lemma \ref{compactificationlemma}, we first recall a result from \cite{Abreu2001} (see also \cite[Lemma 3]{Apostolov2004}).

\begin{theorem}[{\cite[Theorem 2]{Abreu2001}}] \label{abreuthm}
Let $(M,\omega,\G)$ be the toric symplectic manifold associated to a labelled Delzant polytope $(\polytope,\L)$, and $J$ any $\omega$-compatible toric complex structure. Then $J$ is determined in moment-angle coordinates $(x,\tt)\in M^o\cong\polytope^0\times\G^n$ by
\begin{equation*}
\begin{cases} 
       	Jd\tt=&-\langle \GG^u(x),dx\rangle\\
		Jdx=&-\langle \HH^u(x),d\tt\rangle
\end{cases}
\end{equation*}
in terms of a symplectic potential $u\in\symppot$ of the form
\begin{equation}\label{star}
u=u_\polytope+h,
\end{equation}
where $u_\polytope=\frac{1}{2}\sum_r L_r\log(L_r)$ is the so-called Guillemin potential, $h$ is a smooth function on the whole of $\polytope$, the matrix $\HH^u=(\GG^u)^{-1}$ with $\GG^u=Hess(u)$ positive definite on $\polytope^o$ and having determinant of the form 
\begin{equation}\label{2star}
det(\GG)=\frac{\delta(x)}{\Pi_rL_r(x)},
\end{equation}
where $\delta$ is a smooth and strictly positive function on the whole $\polytope$.\\
Conversely any such $u$ determines a compatible toric complex structure on $(M,\omega)$, which in suitable $(x,\tt)$ coordinates of $\polytope^o\times\G^n$ has the form 
\begin{equation*}
\begin{cases} 
       	Jd\tt=&-\langle \GG^u(x),dx\rangle\\
		Jdx=&-\langle \HH^u(x),d\tt\rangle
\end{cases}
\end{equation*}
\end{theorem}

\begin{remark}
The arguments in \cite[Proposition 1]{Apostolov2004} show that, more generally, (\ref{star}) and (\ref{2star}) are equivalent with the defining smoothness, positivity, and boundary conditions (see (\ref{toriccond1}) above) of $\symppot$, independent of the integrality of $(\polytope,\L)$.
\end{remark}

We will also need the following

\begin{lemma}{\cite{apostolov_cr_2020}} \label{projhessian}
Let $u\in\symppot$ and consider $f(x)=a_0+\sum_{i=1}^da_ix_i$ an affine linear function which is positive on $\polytope$ containing the origin. If $\tu(\tx)=\frac{u(x)}{f(x)}$, then $\GG=Hess_x(u)$ and $\tGG=Hess_{\tx}(\tu)$ are related by $$\det(\tGG)=\frac{f^{m+2}(x)}{a_0^2}\det \GG. $$
\end{lemma}
\begin{proof} This follows from \cite[Lemma 5]{apostolov_cr_2020}. For the sake of completeness, we present here a direct argument in the case $m=2$.

Let $f(x)=a_0+a_1x_1+a_2x_2$. We have $x_i=\frac{\tx_i}{\tf(\tx)}$ where $\tf(\tx)=\frac{1}{f(x)}$. Then we obtain:
\begin{align} 
\begin{split} \label{xtotx}
\tpartial_k(x_j)=&\frac{\partial x_j}{\partial\tx_k}=\frac{\partial}{\partial \tx_k}\left(\frac{\tx_j}{\tf(\tx)}\right)\\
=&\frac{\delta_{kj}\tf(\tx)-\ta_k\tx_j}{\tf^2(\tx)}\\=&f(x)\left(\delta_{kj}f(x)+\frac{a_k}{a_0}x_j\right),
\end{split}
\end{align}
 where $k,j=1,2$. Also, we observe that: 
\begin{align}
\begin{split} \label{tpartials}
\tpartial_{ij}=&\frac{\partial^2}{\partial\tx_i\partial\tx_j}\\
=&\frac{\partial}{\partial\tx_i}\left( \frac{\partial x_1}{\partial\tx_j}\frac{\partial}{\partial x_1}+\frac{\partial x_2}{\partial\tx_j}\frac{\partial}{\partial x_2}\right) \\
=& \left(\frac{\partial x_1}{\partial\tx_i}\frac{\partial}{\partial x_1}+\frac{\partial x_2}{\partial\tx_i}\frac{\partial}{\partial x_2}\right)\circ\left( \frac{\partial x_1}{\partial\tx_j}\frac{\partial}{\partial x_1}+\frac{\partial x_2}{\partial\tx_j}\frac{\partial}{\partial x_2}\right).
\end{split}
\end{align}
Now, using (\ref{xtotx}) and (\ref{tpartials}) we obtain the following formulas for $\tu_{,ij}(\tx)=\tpartial_{ij}\tu(\tx)$:

\begin{align}
\begin{split} \label{tuhessian}
\tu_{,11}(\tx)=& \frac{f(x)}{a_0^2}\left( \left((a_1x_1+a_0)^2\frac{\partial^2 u(x)}{\partial x_1^2}+2a_1x_2\left((a_1x_2+a_0) \frac{\partial^2 u(x)}{\partial x_1\partial x_2}+\frac{a_1}{2}x_2\frac{\partial^2 u(x)}{\partial x_2^2} \right)\right) \right)\\
\tu_{,12}(\tx)=& \frac{f(x)}{a_0^2}\left( (a_0f(x)+2a_1a_2x_1x_2)\frac{\partial^2 u(x)}{\partial x_1\partial x_2}+(a_1x_1+a_0)a_2x_1\frac{\partial^2 u(x)}{\partial x_1^2}+ (a_2x_2+a_0)a_1x_2\frac{\partial^2 u(x)}{\partial x_2^2}\right)\\
\tu_{,22}(\tx)=& \frac{f(x)}{a_0^2}\left( \left((a_2x_2+a_0)^2\frac{\partial^2 u(x)}{\partial x_2^2}+2a_2x_1\left((a_2x_1+a_0) \frac{\partial^2 u(x)}{\partial x_1\partial x_2}+\frac{a_2}{2}x_1\frac{\partial^2 u(x)}{\partial x_1^2} \right)\right) \right)\\
\end{split}
\end{align}

Finally, straight forward computation of $\det(\tGG)=\det(Hess(\tu))$ using (\ref{tuhessian}) yields
\begin{equation}\label{dettG}
\det(\tGG)=\frac{f^4(x)}{a_0^2}\det \GG.
\end{equation}
\end{proof}

\begin{proof}[Proof of Lemma \ref{compactificationlemma}]
To prove Lemma \ref{compactificationlemma} we shall check the equivalent conditions for $\tu\in\tsymppot$ given by Theorem \ref{abreuthm}. In order to use Theorem \ref{abreuthm}, the first step is to check that $\tu(\tx)=\tu_{\tpolytope} (\tx)+\tphi(\tx)$, where $\tu_{\tpolytope}$ is the Guillemin potential of $(\tpolytope,\tL)$ and $\tphi$ is a smooth function on $\tpolytope$. Since $u\in\symppot$, we have $u(x)=u_{\polytope}(x)+h(x)$ where $u_{\polytope}$ is the Guillemin potential of $(\polytope, \L)$ and $h$ is a smooth function on the whole of $\polytope$. Now, using that $\tu(\tx)=\frac{u(x)}{f(x)}$ we can write $$\tu(\tx)=\tu_{\tpolytope}(\tx)+\tphi(\tx),$$ where $\tu_{\tpolytope}$ is the Guillemin potential of $(\tpolytope, \tL)$ and $\tphi(\tx)=h\left(\frac{\tx}{\tf}\right)\tf-\log(\tf)\left(\sum_r\tLL_r\right)$. The smoothness of $\tphi$ on $\tpolytope$ follows from the smoothness of $h$ on $\polytope$ and the positivity of $\tf$ on $\tpolytope$.

The second step is to check the positivity of $\tGG$ in $\tpolytope^0$ and its behaviour on $\partial\tpolytope$. The positivity of $\tGG$ on $\tpolytope^o$ follows from \cite[Theorem 1]{apostolov_cr_2020} and \cite[Lemma 5]{apostolov_cr_2020} (which identifies $\tu$ with the simplectic potential of a K\"ahler metric over $\tpolytope^o\times\G^m$).

To check the behaviour of $\det(\tGG)$ on $\partial\tpolytope$ we need to show that $$\det(\tGG)=\frac{\tdelta(x)}{\Pi_r\tLL_r(x)},$$ with $\tdelta$ being a smooth and strictly positive function on the whole $\tpolytope$. This follows from Lemma \ref{projhessian}. Indeed, since $u\in\symppot$ according to Theorem \ref{abreuthm} 
$$det(\GG)=\frac{\delta(x)}{\Pi_{r=1}^dL_r(x)}$$
Using (\ref{dettG}) we obtain,
\begin{align} \label{abreucondsat}
\begin{split}
\det\tGG=&\frac{(f(x))^{m+2}}{a_0^2}\frac{\delta(x)}{\Pi_{r=1}^dL_r(x)}\\
=& \frac{\delta(x)}{a_0^2(f(x))^{d-(m+2)}}\frac{1}{\Pi_{r=1}^dL_r(x)}\\
=& \frac{\tdelta(\tx)}{\Pi_{r=1}^d\tLL_r(x)},
\end{split}
\end{align}
where $\tdelta(\tx)=\frac{1}{a_0^2}\delta\left(\frac{\tx}{\tf}\right)(\tf(\tx))^{d-(m+2)}$ is a positive function on $\tpolytope$.
\end{proof}

\begin{definition}[see \cite{apostolov_cr_2020} p.16 and Lemma 5]  \label{ftwistmetric}
For a toric K\"ahler metric $g$ over ${M^0\cong\polytope^0\times\G^n}$ given in moment-angle coordinates by (\ref{momentanglecoords}), and an affine linear function $f(x)=a_0+\sum_{i=1}^na_ix_i$ positive on $\polytope$ with $a_0>0$, we define the \emph{$f$-twist transform of $g$} to be the toric K\"ahler metric $\tg$ over $\tpolytope\times\G^n$ given by 
\begin{equation}\label{tmomentanglecoords}
    \begin{array}{lcl}
    \tg=\langle d\tx,\tGG(\tx),d\tx\rangle+\langle d\ttt,\tHH(\tx),d\ttt\rangle, & & \tJ d\ttt=-\langle \tGG(\tx),d\tx\rangle, \\
    \tomega=\langle d\tx\wedge d\ttt\rangle, & & \tJ d\tx=-\langle \tHH(\tx),d\ttt\rangle,
    \end{array}
\end{equation}
with
\begin{equation*}
   \tilde{t}_j=t_j-\frac{a_j}{a_0}t_0, \phantom{aa} j\in{1,\ldots,n}, \phantom{aa} \textrm{and}\phantom{aa}\tu(\tx)=\frac{u(x)}{f(x)}. 
\end{equation*}
\end{definition}

\begin{theorem}[{\cite[Theorem 1, Lemma 5]{apostolov_cr_2020}}]
$(\tg,\tJ)$ is extremal if and only if $(g,J,f)$ is $(f,m+2)$-extremal.
\end{theorem}

We complete the above observation with the following

\begin{proposition} \label{keyremark} Let $(\polytope,\L)$ be a simple compact convex labelled polytope in $\Rm$ which contains the origin, and $f(x)$ an affine linear function which is positive on $\polytope$. Consider $(\tpolytope,\tL)$ to be the $f$-twist transform of $(\polytope,\L)$. Then,
$$\wfDonFutinv(\phi)=\frac{1}{f(0)}\tDonFutinv(\tilde{\phi}),$$
where $\frac{\fext(x)}{f(x)}=\extt(\tx)$ and $\tphi(\tx)=\frac{\phi(x)}{f(x)}$.
\end{proposition}

\begin{proof}
Let $T:\polytope\rightarrow\tpolytope$ be the diffeomorphism given by $\tx:=T(x)=\frac{x}{f(x)}$. We consider the Lebesgue measure $dx=dx_1\wedge~\ldots\wedge dx_m$ on $\polytope$ and the induced measures $d\sigma$ on each facet $F_i\subset\partial\polytope$ defined by letting $dL_i\wedge d\sigma=-dx$. In the same way we define $d\tx$ on $\tpolytope$ and $d\tsigma$ on $\tF_i\subset\tpolytope$, respectively.

We observe that:
\begin{equation} \label{measuresrelation}
\begin{aligned}
\T^{\ast}(dx)&=\frac{f(0)}{(\tf(\tx))^{m+1}}d\tx \\
\T^{\ast}(d\sigma)&=\frac{f(0)}{(\tf(\tx))^m}d\tsigma.
\end{aligned}
\end{equation}

Using (\ref{twistpullbackrelation}) and (\ref{measuresrelation}) the result follows. Letting $\tphi(\tx)=\frac{\phi(x)}{f(x)}$, we obtain:
\begin{equation}\label{tscal}
\begin{aligned}
\wfDonFutinv(\phi)=&2\int_{\partial\polytope}\frac{\phi}{f^{m+1}}d\sigma-\int_{\polytope}\frac{\phi}{f^{m+3}}\fext dx\\
  =&\frac{1}{f(0)}\left(2\int_{\partial\polytope}\frac{\phi}{f}\frac{f(0)}{f^m} d\sigma-\int_{\polytope}\frac{\phi}{f}\frac{\fext}{f} \frac{f(0)}{f^{m+1}} dx\right)\\
  =&\frac{1}{f(0)} \left(2\int_{\T(\partial\polytope)}\T^\ast\left(\frac{\phi}{f}\right)\frac{f(0)}{(\T^\ast f)^m} \T^\ast(d\sigma)\right. \\
  &-\left.\int_{\T(\polytope)}\T^\ast\left(\frac{\phi}{f}\right)\T^\ast\left(\frac{\fext}{f}\right) \frac{f(0)}{(\T^\ast f)^{m+1}} \T^\ast(dx)\right)  \\
  =&\frac{1}{f(0)}\left(2\int_{\partial\tpolytope}\tilde{\phi} d\tsigma-\int_{\tpolytope}\tilde{\phi}\frac{\fext}{f}d\tx\right)\\
  =&\frac{1}{f(0)}\tDonFutinv(\tilde{\phi}).
\end{aligned}
\end{equation}

\end{proof}

\begin{corollary} \label{keycorollary}
Let $(\polytope,\L)$ be a labelled compact convex simple polytope in $\Rm$ containing the origin, $f$ a positive affine linear function on $\polytope$, and $(\tpolytope,\tL)$ the $f$-twist of $(\polytope,\L)$. Then, $(\polytope,\L)$ is $(f,m+2)$-$K$-stable if and only if $(\tpolytope,\tL)$ is $K$-stable.
\end{corollary}

\section{Proof of the Main Result}

\subsection{Known Results on Hirzebruch Surfaces} \label{knownres}

Denote by $\mathbb{F}_k$ the $k$-th Hirzebruch complex surface, $\mathbb{F}_k=\mathbb{P}\left(\mathcal{O}\oplus\mathcal{O}(k)\right)\xrightarrow[]{\pi} \mathbb{CP}^1$ for $k\geq1$,
and by $\polytopepk$ the Delzant polytope of the $k$-th Hirzebruch surface $\mathbb{F}_k$ endowed with a $\G^2$-invariant K\"ahler metric in the K\"ahler class $\Omega_p=\mathcal{L}-(1-p)\mathcal{E}$, where $\mathcal{L}$ and $\mathcal{E}$ are respectively the Poincaré duals of a projective line and the infinity section of $\mathbb{F}_k$ (see \cite{gauduchon2009hirzebruch}). It can be shown that the correponding Delzant polytope $\polytopepk$ is the convex hull of $(0,0)$, $(p,0)$, $(p,(1-p)k)$, $(0,k)$, $(0<p<1)$, and labelling $\L_{p,k}=\{L_1,\ldots,L_4\}$ where $e_1=(1,0), e_2=(0,1), e_3=-e_1, e_4=-(ke_1+e_2)$, $\left\langle e_1, x\right\rangle=x_1$, $\left\langle e_2,x\right\rangle=x_2$, and $L_1(x)=\left\langle e_1,x\right\rangle=x_1, L_2(x)=\left\langle e_2,x\right\rangle=x_2, L_3(x)=\left\langle e_3,x\right\rangle+p=-x_1+p, L_4(x)=\left\langle e_4,x\right\rangle+k=k(1-x_1)-x_2$.

In \cite{futaki2018volume}, the authors computed the critical points of the volume functional which characterizes the possible positive affine linear functions $f$ on $(\polytopepk,\Lpk)$  satisfying condition $(a)$ in (\ref{abconditions}), and found out the following

\begin{theorem}[\cite{futaki2018volume, lebrun2016einstein}] \label{foclassfk}
Let $M=\F_k$ be the $k$-th Hirzebruch surface considered as a toric manifold classified by $(\polytopepk, \Lpk)$. Let $0<r_k<s_k<1$ be the real roots of $$F_k(p)=4(1-p)^2k^2-4(p-1)(p-2)pk+p^4.$$

\begin{itemize}
\item[(i)] For any $k$ and $0<p<1$, the affine linear function 
\begin{equation} \label{solution1fk}
f_p=\frac{p+2\sqrt{1-p}-2}{2p^2}x_1-\frac{\sqrt{1-p}-1}{2p}
\end{equation}
is positive on $\polytopepk$ and $(\polytopepk,\Lpk,f_p)$ satisfy the conditions $(a)$ and $(b)$ in (\ref{abconditions}). 

\item[(ii)] For $k=1$ and for $\frac{8}{9}<p<1$, the two affine linear functions 
\begin{equation} \label{solution2fk}
f_p^\pm=\frac{-p\pm\sqrt{9p^2-8p}}{4p^2}x_1+\frac{3}{8}\mp\frac{\sqrt{9p^2-8p}}{8p},
\end{equation}
 are positive on $\polytopepk$ and $(\polytopepk,\L_{p,k},f_p^\pm)$ satisfy the conditions $(a)$ and $(b)$ in (\ref{abconditions}).

\item[(iii)] For $k=1,2,3,4$ and $0<p<r_k$,  let
\begin{displaymath}
a^\pm_{p,k}=\frac{\pm\sqrt{F_k(p)}+2(p-1)k-p(p-2)}{2((2(p-1)(p-2)k-p^3))},
\end{displaymath} 
\begin{displaymath}
b^\pm_{p,k}=\pm\frac{\sqrt{F_k(p)}}{k(2(p-1)(p-2)k-p^3)},
\end{displaymath}
\begin{displaymath}
c^\pm_{p,k}=\frac{1}{4}(1+(p-2)kb^\pm_{p,k}-2pa^\pm_{p,k}),
\end{displaymath}
and consider the two affine linear functions 
\begin{equation} \label{solutions3et4fk}
\fpmpk:= a^\pm_{p,k}x_1+b^\pm_{p,k}x_2+c^\pm_{p,k}.
\end{equation}
The functions $\fpmpk$ are positive on $\polytopepk$ and satisfy the condition $(a)$.
\end{itemize}
\end{theorem}

\begin{remark} \label{remarkfk}
 In \cite{futaki2018volume} the authors showed that the families (\ref{solution1fk}) and (\ref{solution2fk}) of affine linear functions satisfying condition $(a)$ correspond to the Killing potentials of the cKEM metrics constructed in \cite[Theorem D]{lebrun2016einstein} and \cite[Theorem B]{lebrun2016einstein} respectively. Combined with Theorem \ref{am19thm} above, it follows that these families also satisfy the condition $(b)$.
\end{remark}
 
In view of Theorem \ref{foclassfk} and Remark \ref{remarkfk}, the following question arises:

\begin{question}[\cite{futaki2018volume}] \label{q0}
Does the affine linear function given by (\ref{solutions3et4fk}) in Theorem \ref{foclassfk} define a Killing potential for a toric cKEM metric?
\end{question}

In the case of $\fpmpun$ ($k=1$ in (\ref{solutions3et4fk})), numerical evidence towards a positive answer appears in \cite{futaki2019conformally}. Theorem~\ref{thm2} answers Question \ref{q0} in affirmative.

\subsection{Proof of Theorem~\ref{thm2}}

In view of Theorem \ref{equivthm}, we observe that Question \ref{q0} reduces to verify whether or not $(\polytopepk,\L_{p,k},f)$ is $(f,4)$-$K$-stable, i.e. whether or not the condition $(b)$ holds true for $f$ given by (\ref{solutions3et4fk}). By Corollary \ref{keycorollary}, this in turn is equivalent to verify whether or not the $f$-twist $(\tpolytopepk, \tL_{p,k})$ of $(\polytopepk,\L_{p,k},f)$ is $K$-stable.

We are going to show that this is indeed the case, which in turn will yield the existence part of Theorem \ref{thm2}.

To this end, recall the following definition introduced in \cite{legendre2011toric}.

\begin{definition}
Let $\polytope$ be a quadrilateral with vertices $s_1,\ldots,s_4$, such that $s_1$ is not consecutive to $s_3$. We say that a function $f$ is \emph{equipoised} on $\polytope$ if $$\sum_{i=1}^{4}(-1)^if(s_i)=0.$$

A labelled polytope $(\polytope,\L)$ is called \emph{equipoised} if its extremal affine function $\ext$, introduced by (\ref{abreueq}), is equipoised on $\polytope$.
\end{definition}

\begin{theorem}[\cite{Apostolov2015, legendre2011toric}] \label{equipoisedresult}
If $(\polytope,\L)$ is an equipoised labelled compact convex quadrilateral then it is $K$-stable and the Abreu equation (\ref{abreueq}) admits a solution $u\in\symppot$. Furthermore,  the extremal K\"ahler metric corresponding to $\tu$ is either a product, or Calabi-type or an orthotoric metric.
\end{theorem}

\begin{proof}
For the sake of a self-contained presentation we sketch the proof. Following \cite{legendre2011toric}, we recall that given $(\polytope,\L)$ and $g=g_u$ defined by $u\in\symppot$, we say that
\begin{itemize}
\item[$\circ$] $g=g_u$ is of \emph{product type} if $\polytope^0$ admits \emph{product coordinates} $\xi,\eta$ such that on $M^0=\polytope^0\times\G^2$ we have
\begin{equation}\label{productmetric}
g_{|_{M^0}}=\frac{d\xi^2}{A(\xi)}+\frac{d\eta^2}{B(\eta)}+A(\xi)dt_1^2+B(\eta)dt_2^2.
\end{equation}
In this case, the momentum coordinates $x=(x_1,x_2)$ are given by $x_1=\xi$, $x_2=\eta$ and we can assume  $\textrm{Im }_{\polytope^0}\xi=(\alpha_1,\alpha_2)$  and $\textrm{Im }_{\polytope^0}\eta=(\beta_1,\beta_2)$ with $0<\beta_1<\beta_2<\alpha_1<\alpha_2$, and $A\in C^\infty([\alpha_1,\alpha_2])$ and $B\in C^\infty([\beta_1,\beta_2])$ are positive on $(\alpha_1,\alpha_2)$ and $(\beta_1,\beta_2)$, respectively, satisfying the first order boundary conditions 
\begin{equation}\label{boundaryAB}
\begin{aligned} 
&A(\alpha_i)=0=B(\beta_i), \\
&A'(\alpha_1)=r_{\alpha_1}, A'(\alpha_2)=-r_{\alpha_2}, \\
&B'(\beta_1)=r_{\beta_1},  B'(\beta_2)=-r_{\beta_2},
\end{aligned}
\end{equation}
with $r_{\alpha_i}>0,r_{\beta_i}>0$ for $i=1,2$ prescribed by the labelling $\L$.

\item[$\circ$] $g=g_u$ is of \emph{Calabi-type} if $\polytope^0$ admits \emph{Calabi coordinates} $\xi,\eta$ such that on $M^0=\polytope^0\times\G^2$ we have 
\begin{equation}\label{Calabimetric}
g_{|_{M^0}}=\xi\frac{d\xi^2}{A(\xi)}+\xi\frac{d\eta^2}{B(\eta)}+\frac{A(\xi)}{\xi}(dt_1+\eta dt_2)^2+\xi B(\eta)dt_2^2.
\end{equation}
In this case, the momentum coordinates $x=(x_1,x_2)$ are given by $x_1=\xi$, $x_2=\xi\eta$ and we can assume $\textrm{Im }_{\polytope^0}\xi=(\alpha_1,\alpha_2)$  and $\textrm{Im }_{\polytope^0}\eta=(\beta_1,\beta_2)$ with $0<\beta_1<\beta_2<\alpha_1<\alpha_2$, $A\in C^\infty([\alpha_1,\alpha_2])$ and $B\in C^\infty([\beta_1,\beta_2])$ positive on $(\alpha_1,\alpha_2)$ and $(\beta_1,\beta_2)$, respectively, satisfying the first order boundary conditions (\ref{boundaryAB}) at $\alpha_1$, $\alpha_2$ and $\beta_1$, $\beta_2$ (see \cite[Proposition 4.4]{legendre2011toric}).

\item[$\circ$] $g=g_u$ is \emph{orthotoric} if $\polytope^0$ admits \emph{orthotoric coordinates} $\xi,\eta$ such that on $M^0=\polytope^0\times\G^2$ we have  
\begin{multline}\label{orthoricmetric}
g_{|_{M^0}}=\frac{(\xi-\eta)}{A(\xi)}d\xi^2+\frac{(\xi-\eta)}{B(\eta)}d\eta^2\\ +\frac{A(\xi)}{\xi-\eta}(dt_1+\eta dt_2)^2+\frac{B(\eta)}{\xi-\eta}(dt_1+\xi dt_2)^2. \phantom{----------}
\end{multline}
In this case, the momentum coordinates $x=(x_1,x_2)$ are given by $x_1=\xi+\eta$, $x_2=\xi\eta$ and we can assume $\textrm{Im }_{\polytope^0}\xi=(\alpha_1,\alpha_2)$  and $\textrm{Im }_{\polytope^0}\eta=(\beta_1,\beta_2)$ with $0<\beta_1<\beta_2<\alpha_1<\alpha_2$, $A\in C^\infty([\alpha_1,\alpha_2])$ and $B\in C^\infty([\beta_1,\beta_2])$ are positive on $(\alpha_1,\alpha_2)$ and $(\beta_1,\beta_2)$, respectively, satisfying the first order boundary conditions (\ref{boundaryAB}) at $\alpha_1$, $\alpha_2$ and $\beta_1$, $\beta_2$ (see \cite[Proposition 3.1]{legendre2011toric}).
\end{itemize}
We first notice that in \cite{Apostolov2006}, the authors show that for the metrics above to be \emph{extremal}, the functions $A(\xi)$ and $B(\eta)$ must be polynomials of degree $\leq4$ satisfying certain linear relations between their coefficients. We refer to pairs of polynomials $(A(\xi),B(\eta))$ satisfying these relations an \emph{extremal pair} $(A,B)$. \cite[Theorem 1.1]{legendre2011toric} then states that if $(\polytope,\L)$ is an equipoised quadrilateral, one can associate to $(\polytope,\L)$ real numbers $0<\beta_1<\beta_2<\alpha_1<\alpha_2$ and an extremal pair $(A,B)$, verifying the first order boundary conditions (\ref{boundaryAB}), such that they define an extremal K\"ahler metric in $\symppot$, should they be positive on $(\alpha_1,\alpha_2)$ and $(\beta_1,\beta_2)$, respectively. Also, it is shown in \cite[Theorem 1.1]{legendre2011toric} that $(\polytope,\L)$ is $K$-stable if and only if the extremal pair $(A,B)$ is positive on their respective intervals of definition. We now argue that $K$-stability (i.e. positivity of $A$ and $B$) follows automatically from the equipoised condition.

By \cite{legendre2011toric}, if $(\polytope,\L)$ is equipoised, then the solution of the Abreu equation (\ref{abreueq}) (if it exists) must be given by one of the three types described above, according to whether $(\polytope,\L)$ is an equipoised parallelogram, trapezoid which is not a parallelogram, or a quadrilateral which is not a trapezoid, respectively. Furthermore, it is observed \cite{legendre2011toric} that equipoised parallelogram are always $K$-stable and admit extremal K\"ahler metrics of product type. This follows from the boundary conditions (\ref{boundaryAB}) in the product case where an extremal pair $(A,B)$ is defined by the conditions so that $deg A\leq3$ and $deg B\leq3$.

We now consider the Calabi-type case which describes the extremal metrics associated to an equipoised labelled trapezoid which is not a parallelogram. Although the argument does not appear in \cite{legendre2011toric}, the author kindly shared with us in a private communication her observation that any extremal pair $(A,B)$ in this case must also satisfy the positivity assumption (i.e. $(\polytope,\L)$ is $K$-stable if it is an equipoised trapezoid which is not a parallelogram). This follows from the following observation: according to \cite[Proposition 4.6]{legendre2011toric}, a metric of Calabi-type (\ref{Calabimetric}) is extremal if and only if ${A(\xi)=\sum_{i=0}^4a_i\xi^{4-i}}$ has degree at most 4, $B(\eta)$ has degree 2 and $$B''(\eta)=-2a_2=-A''(0).$$ We notice that the boundary conditions (\ref{boundaryAB}) impose that $B(\eta)$ is positive on $(\beta_0,\beta_1)$ which in turn yields $A''(0)=2a_2>0$. If we suppose that $A$ is not positive in $(\alpha_1,\alpha_2)$ this would imply that the two roots of $A''(\xi)$ belong to the interval $(\alpha_1,\alpha_2)$ due to the boundary conditions (\ref{boundaryAB}). However, since $0<\alpha_1<\alpha_2$, for $A''(0)$ to be positive $A''(\xi)$ would have to admit a third root in the interval $(0,\alpha_1)$ which is not possible since $deg A''=2$. Then we conclude that $A(\xi)$ must be positive on $(\alpha_1,\alpha_2)$.

The $K$-stability of an equipoised labelled quadrilateral which is \emph{neither} a parallelogram \emph{nor} a trapezoid was later observed in \cite[Example 1]{Apostolov2015}. This follows from the fact that in this case, $(A,B)$ is an extremal pair if and only if $deg(A+B)\leq1$ \cite[Proposition 3]{Apostolov2015}. Then, between any maximum of $A$ on $(\alpha_1,\alpha_2)$ and of $B$ on $(\beta_1,\beta_2)$, the quadratic $A''=-B''$ has a unique root; the boundary conditions thus force again $A$ and $B$ to be positive on $(\alpha_1,\alpha_2)$ and $(\beta_1,\beta_2)$, respectively.
\end{proof}

\begin{proof}[Proof of Theorem \ref{thm2}]Thus, by virtue of Theorem \ref{foclassfk}, Theorem \ref{equivthm}, Corollary \ref{keycorollary} and Theorem \ref{equipoisedresult} (in that order), the existence part of Theorem \ref{thm2} will be established if we check that the $f$-twist of $(\polytopepk,\L_{p,k})$ with $f$ being the affine linear function given by one of the families (\ref{solution1fk}), (\ref{solution2fk}) or (\ref{solutions3et4fk}) of Theorem \ref{foclassfk} is equipoised. The verification is straightforward in all cases, so we present below only the case (\ref{solutions3et4fk}) (in which the validity of the condition (b) was previously unknown).

We first notice that as $\fextpmk=c$ by Theorem \ref{foclassfk}, we have $\exttk=\frac{c}{\fpmpk}$ by Proposition \ref{keyremark}. It follows that $(\tpolytopepk,\tL)$ is equipoised if and only if
\begin{equation}\label{equipoisedcond}
\sum_{i=1}^4(-1)^i\frac{1}{\fpmpk(s_i)}=0.
\end{equation}

The verification of (\ref{equipoisedcond}) is straightforward and we detail below the computation in the case $k=1$ (for other values of $k$ the computation is similar), and we drop the index $k$ to ease the notation.

\begin{multline} \label{e1}
\sum_{i=1}^{4}(-1)^i\frac{1}{f^\pm_p(s_i)}= 
\frac{1}{p^2+2p-2+\sqrt{F(p)}}+\frac{1}{-p^3+3p^2-4p+2-(1-p)\sqrt{F(p)}}+\\ \frac{1}{p^3-3p^2+4p-2-(1-p)\sqrt{F(p)}}+\frac{1}{-p^2-2p+2+\sqrt{F(p)}}
\end{multline}

If we write $U=p^2+2p-2$, $W=\sqrt{F(p)}$ and $V=p^3-3p^2+4p-2$, the RHS of (\ref{e1}) is given by

\begin{multline}\label{e2}
\frac{1}{U+W}+\frac{1}{-V-(1-p)W}+ \frac{1}{V-(1-p)W}+\frac{1}{-U+W}=\\ 
 \frac{(-V-(1-p)W)(V-(1-p)W)(-U+W)+(U+W)(V-(1-p)W)(-U+W)}{(U+W)(-V-(1-p)W)(V-(1-p)W)(-U+W)}+\\ \frac{(U+W)(-V-(1-p)W)(-U+W)+(U+W)(-V-(1-p)W)(V-(1-p)W)}{(U+W)(-V-(1-p)W)(V-(1-p)W)(-U+W)}=\\
 \frac{-2W\left[V^2+(1-p)(pW^2-U^2)\right]}{(U+W)(-V-(1-p)W)(V-(1-p)W)(-U+W)}
\end{multline}

Now replacing $U,$ $V$, $W$, and $F(x)=x^4-4x^3+16x^2-16x+4$, we can check that $V^2+(1-p)(pW^2-U^2)=0$ in (\ref{e2}). We have performed similar verifications for any $k$.

For the last claim of Theorem~\ref{thm2} see the proof of \cite[Theorem 5]{Apostolov2019conformally} and \cite[Sec. 5.4]{Apostolov2015}.
\end{proof}

\subsection{Proof of Theorem \ref{thm1}}

The first part of Theorem~\ref{thm1} follows directly from Theorem~\ref{thm2}, since the positive affine linear functions giving rise to new cKEM metrics are obtained by taking $k=1$ in (\ref{solutions3et4fk}).

We will thus establish below the uniqueness statement in Theorem \ref{thm1}. To this end, we use  the fact that any K\"ahler metric $(g, \omega)$ on $\F_1$, which is conformal to an Einstein-Maxwell metric $\tg = \frac{1}{f^2} g$,   is invariant under the action of a maximal torus in the automorphism group $Aut(\F_1)$ of $\F_1$ \cite{futaki2020existence, lahdili2019kahler}. As any two maximal tori are conjugated by an element of $Aut(\F_1)$, by acting with such an element on $(g, \omega)$ we can assume that $(g, \omega)$ is invariant under a fixed torus $\G^2 \subset Aut(\F_1)$, and by acting with a homothety on $(g, \omega)$, that the momentum map of $\G^2$ with respect to $\omega$ is a Delzant polytope $\polytope_{p,1}$  (for some $p\in (0,1)$) as defined in Section \ref{knownres}. Now by \cite[Theorems 3 and 5]{Apostolov2019conformally}, the isometry classes of $\G^2$-invariant conformally Einstein--Maxwell K\"ahler metrics  $(g, \omega)$ in the cohomology class determined by  $\polytope_{p,1}$ are in a bijective correspondence with the positive affine linear functions $f$ on $\polytope_{p,1}$, normalized by the condition that sum of $f$ over the vertices of $\polytope_{p,1}$ equals to $1$,  for which the conditions $(a)$ and $(b)$ in (\ref{abconditions}) hold true.

In \cite{futaki2018volume}, the determination of (normalized) solutions $f$ to the  condition (a) of (\ref{abconditions}) is studied in detail. Additionally to the solutions listed in Theorem \ref{foclassfk} above, the authors found in \cite{futaki2018volume} two explicit families of normalized affine linear functions $\fpmb(x_1,x_2)$ which would verify the condition (a) of (\ref{abconditions}) should they be positive on $\polytope_{p,1}$. But this last point was left open, see \cite[p. 26]{futaki2018volume}

\begin{proposition}[\cite{futaki2018volume}] \label{focases1and2} Let $M=\F_1$ be the first Hirzebruch surface classified by the Delzant polytope $(\polytope_{p,1}, \L_{p,1})$, $p\in(0,1)$, introduced in Section \ref{knownres}. Letting 
$$E_b(p)=b^2(1-p)(2-3p)^2+p^2+p-1,$$ 
consider the two affine linear functions 
$$f^\pm_b:=a^\pm_bx_1+bx_2+c^\pm_b,$$ where
\begin{displaymath}
a^\pm_b=\frac{3bp^2+(1-2b)p\pm\sqrt{E_b(p)}}{2p(3p-2)},
\end{displaymath} 
\begin{displaymath}
c^\pm_b=\frac{1}{4}\left( 1-(2-p)b-2pa^\pm_b\right),
\end{displaymath}
which are defined for any value $b\in\mathbb{R}$ such that 
\begin{equation}\label{ineq1}
b^2\geq\frac{1-p-p^2}{(1-p)(2-3p)^2}
\end{equation}
for a fixed $p\in(0,1)$, $p\neq\frac{2}{3}$. If $\fpmb$ is positive on $\polytope_{p,1}$, then it satisfies condition $(a)$ in (\ref{abconditions}).
\end{proposition}

Thus, in order to obtain the uniqueness statement of Theorem \ref{thm1}, it is enough to show that the affine linear functions $\fpmb$ are not positive on $\polytope_{p,1}$.

\begin{lemma} \label{cases12ruledout}
The affine linear functions $\fpmb$ defined in Proposition \ref{focases1and2} are not positive on $\polytope_{p,1}$.
\end{lemma}
\begin{proof}[Proof of Lemma \ref{cases12ruledout}]

We know that an affine linear function is positive over a convex polytope if $f(s)>0$  for every vertex $s$ of the polytope. Let $s_1=(0,0)$, $s_2=(0,1)$, $s_3=(p,0)$ and $s_4=(p,1-p)$ be the vertices of $\polytope_{p,1}$ with $p\in(0,1)$. 

We compute that $\fpmb$ is positive on $\polytope_{p,1}$ if and only if:
\begin{equation} \label{fsi}
\begin{aligned}
\fpmb(s_1)&=\frac{-(1-p)+(2-3p)b\mp\sqrt{E_b(p)}}{2(3p-2)}>0\\
\fpmb(s_2)&=\frac{-(1-p)-(2-3p)b\mp\sqrt{E_b(p)}}{2(3p-2)}>0\\
\fpmb(s_3)&=\frac{(1-p)(2-3p)b+(2p-1)\pm\sqrt{E_b(p)}}{2(3p-2)}>0\\
\fpmb(s_4)&=\frac{-(1-p)(2-3p)b+(2p-1)\pm\sqrt{E_b(p)}}{2(3p-2)}>0\\
\end{aligned}
\end{equation}
and we observe that $\fpmb(s_1)=\fpmbm(s_2)$ and $\fpmb(s_3)=\fpmbm(s_4)$. Because of this symmetry, from now on we consider only $b\geq0$. To ease the notation, we drop the index $b$ and study the positivity of $f^+$ and $f^-$ separately. 

\textbf{Case 1:} Positivity of $f^+$

We first observe that if $p\in\left(\frac{2}{3},1\right)$ then $f^+(s_2)<0$, so we need only to consider $p\in(0,\frac{2}{3})$. For $p\in\left(0,\frac{2}{3}\right)$ and $b\geq0$, the conditions (\ref{fsi}) imply that 
\begin{equation}\label{ineq2}
-(1-p)+(2-3p)b<\sqrt{E_b(p)}<-(1-p)(2-3p)b-(2p-1)
\end{equation}
and the RHS of (\ref{ineq2}) forces
\begin{equation} \label{ineq3}
0\leq b<\frac{(1-2p)}{(1-p)(2-3p)}.
\end{equation}
 Under this assumption the LHS of (\ref{ineq2}) is always negative since $$\frac{(1-2p)}{(1-p)(2-3p)}<\frac{(1-p)}{(2-3p)}.$$
 
Thus the positivity of $f^+$ is equivalent to $p\in\left(0,\frac{1}{2}\right)$, $|b|<\frac{(1-2p)}{(1-p)(2-3p)}$ and 
\begin{equation}\label{ineq4}
E_b(p)<((1-2p)-(1-p)(2-3p)b)^2.
\end{equation}

Now, (\ref{ineq1}) and (\ref{ineq3}) give 
\begin{equation}\label{ineq4.1}
\frac{1-p-p^2}{(2-3p)^2(1-p)}\leq b^2<\frac{(1-2p)^2}{\left[(1-p)(2-3p)\right]^2},
\end{equation}
which leads to
\begin{equation} \label{ineq5}
(1-p)(1-p-p^2)<(1-2p)^2,
\end{equation}
so we must have $p(p^2-4p+2)<0$. However, for $p\in\left(0,\frac{1}{2}\right)$ we have ${p^2-4p+2>0}$, showing that $f^+$ cannot be everywhere positive on $\polytope_{p,1}$.

\textbf{Case 2:} Positivity of $f^-$

Once again we assume without loss of generality $b\geq0$ and consider the two cases $p\in\left(0,\frac{2}{3}\right)$ and $p\in\left(\frac{2}{3},1\right)$.

We consider first the case $p\in\left(0,\frac{2}{3}\right)$. Here, (\ref{fsi}) for $f^-$ reduces to (\ref{ineq1}),
\begin{equation}
b^2\geq\frac{1-p-p^2}{(1-p)(2-3p)^2}
\end{equation}
 and 
\begin{equation}\label{ineq6}
(2p-1)+(1-p)(2-3p)b<\sqrt{E_b(p)}<(1-p)-(2-3p)b.
\end{equation}
For the RHS of (\ref{ineq6}) to be positive we need 
\begin{equation}\label{ineq7}
b<\frac{(1-p)}{(2-3p)}.
\end{equation}
Notice that the LHS of (\ref{ineq6}) is positive if $p\in(\frac{1}{2},\frac{2}{3})$. In this case, 
\begin{equation}
E_b(p)>((2p-1)+(1-p)(2-3p)b)^2.
\end{equation}
which is equivalent to 
\begin{equation}\label{ineq9}
0<E_b(p)-((2p-1)+(1-p)(2-3p)b)^2=(1-p)(2-3p)(bp+1)(b(2-3p)-1),
\end{equation}
and for the RHS of (\ref{ineq9}) to be positive we need $b>\frac{1}{(2-3p)}$ which contradicts the inequality (\ref{ineq7}).

On the other hand, if $p\in\left(0,\frac{1}{2}\right)$, we compare inequalities (\ref{ineq7}) and (\ref{ineq1}) and, as in the case of $f^+$ we obtain the inequality (\ref{ineq4.1}). This leads to ${p(p^2-4p+2)<0}$ and this is impossible for $p\in\left(0,\frac{1}{2}\right)$. So we conclude that $f^-$ is negative somewhere on $\polytope_{p,1}$ whenever $p\in\left(0,\frac{2}{3}\right).$

Now, we move to the study of $f^-$ for $p\in\left(\frac{2}{3},1\right)$. Here, conditions (\ref{fsi}) imply
\begin{equation}\label{ineq10}
(1-p)+(3p-2)b<\sqrt{E_b(p)}<(2p-1)-(1-p)(3p-2)b.
\end{equation}
The RHS of (\ref{ineq10}) gives the necessary condition $b\in\left[0,\frac{(2p-1)}{(1-p)(3p-2)}\right)$ and, under these restrictions, both sides of (\ref{ineq10}) are positive. Thus, (\ref{ineq10}) is equivalent to the inequalities
\begin{equation}\label{ineq11}
E_b(p)<((2p-1)-(1-p)(3p-2)b)^2
\end{equation}
and
\begin{equation}\label{ineq12}
E_b(p)>((1-p)+(3p-2)b)^2.
\end{equation}
The condition (\ref{ineq11}) is similar to the condition (\ref{ineq4}) for $f^+$ and we obtain, from the factorization presented there, that 
\begin{equation}\label{ineq13}
(1-p)(2-3p)(bp+1)(b(2-3p)-1)<0.
\end{equation}
So (\ref{ineq13}) gives $b(2-3p)-1>0$, which is impossible since $(2-3p)<0$ and $b\geq0$. Then $f^-$ needs to be negative somewhere on $\polytope_{p,1}$ whenever $p\in\left(\frac{2}{3},1\right)$.
\end{proof}

\begin{remark} While \emph{all} possible positive affine linear functions on $\polytope_{p,1}$ satisfying the condition $(a)$ in (\ref{abconditions}) are determined by virtue of \cite{futaki2018volume} and Lemma \ref{cases12ruledout} above, the question remains open for $k\geq2$. As a matter of fact, it is still unknown whether or not for $k=2,3,4$ the affine linear functions given by (\ref{solution1fk}), (\ref{solution2fk}) and (\ref{solutions3et4fk}) are the \emph{only} such functions whereas for $k\geq5$ it is unknown even wheter or not the affine linear function defined in (\ref{solutions3et4fk}) is actually positive over $\polytope_{k,p}$.
\end{remark}

\bibliography{BibitexIsaque.bib}
\bibliographystyle{acm}

\end{document}